\documentclass[11pt]{amsart}

\usepackage{amsmath}
\usepackage{amssymb,amsthm,amsfonts}
\usepackage{amsrefs}
\usepackage[utf8]{inputenc}
\usepackage[T1]{fontenc}
\usepackage{xcolor}
\usepackage[colorlinks]{hyperref}

\newcommand{\Z}{{\mathbb Z}}
\newcommand{\N}{{\mathbb N}}
 \newcommand{\G}{\mathcal{G}}
  \newcommand{\D}{\mathcal{D}}
\newcommand{\R}{{\mathbb R}}

\renewcommand{\phi}{{\varphi}}

\newcommand{\Per}{\mathrm{Per}}

\newtheorem{theorem}{Theorem}[section]
\newtheorem{corollary}[theorem]{Corollary}

\newtheorem{lemma}[theorem]{Lemma}
\newtheorem{proposition}[theorem]{Proposition}

\theoremstyle{definition}
\newtheorem{definition}[theorem]{Definition}
\newtheorem{remark}[theorem]{Remark}

\begin{document}
 
\title{Minimal periodic foams with equal cells}

\author{Annalisa Cesaroni}
\address{Dipartimento di Matematica "Tullio Levi-Civita", Universit\`{a} di Padova, Via Trieste 63, 35131 Padova, Italy}
\email{annalisa.cesaroni@unipd.it}
\author{Matteo Novaga}
\address{Dipartimento di Matematica, Universit\`{a} di Pisa, Largo Bruno Pontecorvo 5, 56127 Pisa, Italy}
\email{matteo.novaga@unipi.it}

\begin{abstract} 
We show existence of  periodic foams with equal cells in $\R^n$ minimizing an anisotropic perimeter. 
\end{abstract} 
  
\subjclass{ 
49Q05  
58E12  
74E10  
}
 
\keywords{Minimal partitions, Kelvin cell, $24$-cell, fractional perimeter}
 
\maketitle 
 
\tableofcontents
 
\section{Introduction}
The Kelvin problem, posed by Lord Kelvin in \cite{thompson} (see also \cites{we,kelvin}) is the problem of finding a partition of $\R^3$ into cells
of equal volume, so that the total area of the surfaces separating them is as small as possible.

In this paper we deal with a related  problem in $\R^n$,  that is, finding the minimal foam among all periodic partitions generated by a lattice tiling of $\R^n$. 
A lattice $G$ is a discrete subgroup of $\R^n$ of full rank, and a lattice tiling is given by  $D+G=\R^n$, where $D$ is a fundamental domain associated to the action of $G$. 

The isoperimetric problem can be stated as follows:  among all fundamental domains associated to   lattices of $\R^n$ with fixed volume (where the volume of the lattice is the volume of each  fundamental domain)  show that there exists a domain with minimal perimeter, where the perimeter functional is either the classical isotropic perimeter or an anisotropic one.  The solution of this problem provides a periodic tiling of $\R^n$ with equal cells of minimal perimeter, among all periodic tilings generated by a lattice. 

The starting point of our analysis is the existence of a minimal fundamental domain for a given lattice, which has been proved in great generality 
in \cite{cnparti}, extending previous results in \cites{choe,mnpr}.
Our first result is a compactness property of lattices with volume equibounded from below and with an upper bound on the perimeter of a fundamental domains (see Proposition \ref{Fubini}), which is based on classical compactness results for lattices, such as the Mahler's compactness Theorem \cite{ma} and its generalizations (see \cite{cassels}). The main tool to prove existence of a minimizer is a concentration compactness argument, which is a classical method in order to deal with loss of mass  at infinity in isoperimetric problems in noncompact spaces.
 
Once we show existence of minimal periodic partition with equal cells, an interesting question is analyzing the possible explicit structure of these partitions. This problem is quite hard, and it has been solved only in dimension $2$ by Hales \cite{hales}, whereas it remains open in all the other cases. 

In the last section we describe some candidate minimizers in dimension $3,4,8,24$,  motivated by the analogy with related problems such as the optimal sphere packing. Moreover, in the isotropic case we provide some estimates on the asymptotic behavior of the isoperimetric function  $c(n)$ (that is, the value of \eqref{iso} for $m=1$), as the dimension $n$ goes to $+\infty$, and we show that
\[
\sqrt{2\pi e n}\sim n\omega_n^{\frac{1}{n}}\leq c(n)\leq   \frac{ 2 }{(2\zeta(n))^{\frac{1}{n}}}  n \omega_n^{\frac{1}{n}} \sim 2 \sqrt{2\pi e n},
\]
where $\zeta(n)$ is the Riemann zeta-function. 
Note that one side of the estimate is obtained by direct comparison with    the perimeter of the ball of volume $1$, whereas the other side is based on a noncostructive existence result  due to  Minkowski and Hlawka, see Proposition \ref{asi} for more details. 

Finally, we conclude by observing that similar results are expected to hold also for more general perimeter functional, such as the nonlocal perimeters. 
This will be the subject of future investigations.  

\smallskip
{\it Acknowledgements.} The authors are members of INDAM-GNAMPA. The second author was supported by the PRIN Project 2019/24.

\section{Lattices}\label{sec:lat} 
We recall the definition of lattice in $\R^n$.

\begin{definition}[Lattice] 
A lattice is a discrete subgroup $G$ of $(\R^N, +)$ of rank $n$. 
The elements of $G$ can be expressed as $\sum k_iv_i$, for a given basis  $(v_1, \dots, v_n)$ of $\R^N$, with    coefficients $k_i\in \Z$.   Any two bases for a lattice $G$ are related  by a matrix with integer coefficients and determinant equal to $\pm1$.  

The absolute value of the determinant of the matrix of any set of generators $v_i$ is uniquely determined and it is equal to $d(G)\in (0,+\infty)$, which we call {\it volume of the lattice }$G$.
Equivalently, every lattice can be viewed as a  discrete group of  isometries of $\R^N$, and $d(G)$ coincides with the volume of the quotient torus $\R^N/G$. 
Given $m\in (0,+\infty)$, we denote by $\G_m$ the set of all lattices  $G$  such that $d(G)=m$.

We define the {\it minimum distance} $\lambda(G)>0$ in the lattice $G$ as the length of the shortest nonzero element of $G$.
In particular, for every $p,q\in G$, there holds that $|p-q|\geq \lambda (G)$. 
Other important values associated to a lattice $G$  are its {\it packing radius} $\rho_G$  and its \it{covering radius} $r_G$ defined  as 
\begin{eqnarray*} 
\rho_G&:=&\sup\{ r\ :\  \forall x\neq y\in G, B_r(x)\cap B_r(y)= \emptyset\}= \frac{\lambda(G)}{2}\\\nonumber
 r_G&:=&\inf\{ r\ :\  G+B_r=\R^N\}.\end{eqnarray*}
 
\end{definition} 

We also recall the definition of Voronoi cell associated to a lattice.

\begin{definition}[Voronoi cell]\label{voronoi} Given a lattice $G$ of $\R^n$, we define the Voronoi cell associated to $G$ as  
\[V_G:=\{x\in \R^n :\  |x|\leq |x-g| \qquad \forall g\in G, g\neq 0\}. \]
 $V_G$ is a   centrally symmetric and convex polytope, with at most  $2(2^n-1)$ facets. Moreover $V_G+G=\R^n$ and $V_G+g\cap V_G+h=\emptyset$ for $g\neq h, g,h\in G$. 
\end{definition} 

\begin{lemma} \label{lemmaretta} 
A closed subgroup $G$ of $(\R^n, +)$ is discrete if and only if it does not contain a line.
\end{lemma}

\begin{proof} 
Clearly, if $G$ is discrete, then it cannot contain a line.

Assume that $G$ is not discrete. Then there exists a sequence 
 $g_n$ converging to some $g\in G$, with $g_n\ne g$ for all $n$
 and such that $\frac{g_n-g}{|g_n-g|}\to e$, where $e$ is a unitary vector of $\R^n$. 
 Fix $r\in \R$ and, for all $n$, fix $k_n\in\Z$ such that 
 $\left| k_n |g_n-g|-r\right|\leq |g_n-g|$. Then we have 
 \[|k_n (g_n-g)- r e|\leq |g_n-g|+ |r| \left|\frac{g_n-g}{|g_n-g|}-e\right|,
 \] 
so that the line $\R e$ lies in the closure of $G$. \end{proof} 

The following result shows the existence of a reduced set of generators for lattices (see \cite[Theorem 1]{ma}).

 \begin{lemma}\label{limi} There exists a dimensional constant $C_n$ such that every lattice $G$ admits a set of generators $v_1, \dots, v_n$ 
 with $\Pi_{i=1}^n|v_i|\leq C_n d(G)$. 
 \end{lemma} 
 
 We introduce a notion of convergence of lattices see \cites{cassels, ma}.  
 Note that, if $G$ is a lattice, then for every compact  subset $K\subset\R^n$, the set $G\cap K$ is finite.

 \begin{definition} 
 A sequence of lattices $G_k$ converges to $G$ if there exists for all $k$ a set of generators $g_k^i$ of $G_k$ such that $g_k^i\to g^i$, and $g^i$ is a set of generators of the lattice $G$.

A sequence of lattice $G_k$ converges in the Kuratowski sense to $G_i$, if
\[G=\{g\in \R^n \ | \limsup_{i\to +\infty} d(g, G_i)=0\}.\]
Note that $G$ is a closed subgroup of $(\R^n, +)$. 
Actually, the two notion of convergence are equivalent, see \cite[Section V.3, Theorem 1]{cassels}. 
 \end{definition} 
 
We recall the following compactness theorem for lattices due to Mahler \cite[Theorem 2]{ma} (see also \cite[Chapter V]{cassels}). 

\begin{lemma}\label{lemmaconvergence} 
Let $G_i$, $i\in \N$,  be a sequence of lattices and assume that there exist two constants $k,\delta>0$ such that  $\lambda(G_i)\geq \delta>0$ for all $i$ and $d(G_i)=|\R^n/G_i|\geq k>0$ for all $i$.   Then there exists a subsequence $G_{i_n}$ and a lattice $G$ such that $G_{i_n}\to G$, and $\lambda(G)\geq \delta$, $d(G)\geq k$.
\end{lemma}
 

We recall the notion of fundamental domain for the action of a group  $G$.

\begin{definition}[Fundamental domain] We say that $D\subset\R^n$ is a fundamental domain for the group $G$ 
if it is a fundamental domain for the action of $G$ on $\R^n$, that is  
a set which contains almost all representatives for the orbits of $G$ and such that the points
whose orbit has more than one representative has measure zero.  
 We denote by $\D_G$ the set of all fundamental domains of $\R^n$  for the group $G$.  Notice that $|D|=|\R^n/G|$ for all $D\in\D_G$. 
  \end{definition} 
\begin{remark}[Fundamental domains associated to lattices]\upshape \label{remcomp} If $G$ is a lattice, then its Voronoi cell is a fundamental domain. Moreover if $v_1, \dots, v_n$ is a set of generators of 
 $G$ as a $\Z$-module, then  the set $D=\{x=\sum_{i=1}^n t_iv_i, \ t_i\in [0,1)\}$ is a fundamental domain associated to $G$. Indeed it is easy to check, using the fact that  $v_1, \dots, v_n$ is a vectorial basis of $\R^n$, that every element in $\R^n$ can be written as   an element in $D$  translated by an element of the group $G$.
   \end{remark} 
 
We now recall the notion of tiling  of $\R^n$. 

\begin{definition}
A partition or tiling of $\R^n$ is a collection of measurable subsets $\{E_k\}_{k\in\mathbb I}$,  where $\mathbb I$ is either a finite or a countable set of ordered indices, such that 
\begin{enumerate} 
\item $|E_k|>0$ for all $k$,
\item $|E_k\cap E_j|=0$ for all $k\ne j$,
\item $|\R^n\setminus\cup_k E_k|=0$.
\end{enumerate}
\end{definition}
Each fundamental domain $D$ associated to a lattice $G$ induces a  lattice  tiling  of $\R^n$: that is $\R^n=\cup_{g\in G} D+g$.
 Moreover it is possible to prove that  the fundamental domain $D$ is precompact if and only if the $G$-periodic partition $\{D+g\}_{g\in G}$ induced by $D$ is locally finite (see \cite{cnparti}).

 \section{Existence of minimal partitions}\label{sec:ex}
 Let us fix a   spatially homogeneous norm $\phi$ in $\R^n$,   and consider the local (anisotropic) perimeter associated to $\phi$: for every measurable set $E\subseteq\R^n$, we define 
\[\Per_\phi(E)=\int_{\partial^* E} \phi(\nu(x))dH^{n-1}(x)\] where $\partial^* E$ is the reduced boundary of $E$. 
When $\phi(x)=|x|$, we recover the classical perimeter $\Per(E)$. 

We will show existence of a solution of the following isoperimetric problem: let  $m>0$ and consider 
 \begin{equation}\label{iso} \inf_{G\in\G_m} \inf_{D\in\D_G} \Per_\phi(D).
 \end{equation} 
 Observe that if $D$ is a fundamental domain associated to a lattice $G\in \G_m$ which solves the previous isoperimetric problem, then the $G$-periodic partition $\{E_g\}_{g\in G}$, generated by $D$, that is defined as  $E_g = g+ D$, is a partition of equal cells in $\R^n$ with minimal  perimeter among all  partitions of equal cells. 

We start with a compactness result.  
 \begin{proposition}\label{Fubini} 
 Let $G_i$, for $i\in\N$,  be a sequence of lattices in $\G_m$ and let $D_i$ be a fundamental domain for $G_i$.
 If there exists $C>0$ such that $\Per_\phi(D_i)\leq C$ for all $i$, then  there exists $\delta=\delta(C,\phi,m)>0$ such that $\lambda(G_i)\geq \delta$ for all $i$. 
 
 In particular, up to a subsequence, $G_i\to G$ in Kuratowski sense, where $G$ is a lattice in $\G_m$. 
 \end{proposition} 
 
\begin{proof} 
It is sufficient to prove the result for the classical (isotropic) perimeter, 
since a uniform upper bound on $\Per_\phi(D_i)$ implies a uniform upper bound on $\Per(D_i)$. 

For every $G_i$ we consider a vector $g_i$ of minimal norm, that is, an element of $G_i$ such that $|g_i|=\lambda(G_i)$. Let $V_i$ be the orthogonal subspace to $g_i$, and let $\pi_i:\R^n\to V_i$ the orthogonal projection. Notice that, since $|D_i\cap (D_i+g_i)|=0$, every line of the type $\{x+tg_i:\,x\in\R^n,\,t\in \R\}$ intersects $D_i$ 
in a set of measure at most $|g_i|=\lambda(G_i)$. Then,
by Fubini-Tonelli Theorem  we have that  $m=|D_i|\le\lambda(G_i)H^{n-1}(\pi_i(D_i))$. 
On the other hand, from the area formula it follows that  $\Per(D_i)\geq 2\,H^{n-1}(\pi_i(D_i))$. 
Therefore we get 
$$\lambda(G_i)\ge \frac{m}{H^{n-1}(\pi_i(D_i))}\geq \frac{2m}{\Per(D_i)}\geq \frac{2m}{C}.$$ 
The conclusion is now a direct consequence of Lemma \ref{lemmaconvergence}. 
\end{proof} 

We recall now the  concentration compactness result which is, together with   lower semicontinuity and compactness properties of the perimeter functional, the standard tool to show existence in isoperimetric problems, since it permits to deal with the possible loss of mass at infinity of minimizing sequences. 
For the proof we refer to \cite[Lemma 3.3, Lemma 3.4]{cnparti}, see also \cite[Theorem 3.3]{npst}.

\begin{lemma}[Concentration compactness]\label{lemmacc}   Let 
$E_k\subseteq \R^n$ be a sequence of measurable sets with   $\sup_k \Per_\phi(E_k)\leq C<+\infty$ and $|E_k|=m$, for some $m, C>0$. 

Then, up to passing to a subsequence,  there exists a sequence  $z_k^i\in \R^n$, $i\in\N, k\in\N$,  with $|z^i_k-z_k^j|\to +\infty$ as $k\to +\infty$ for $i\neq j$ and a family $(E^i)$ of measurable sets in $\R^n$ such that 
\begin{enumerate}
\item  $E_k-z_k^i\to E^i$ locally in $L^1$ as $k\to +\infty$, for all $i$;
\item $\sum_i\Per_\phi(E^i)\leq  \liminf_k\Per_\phi(E_k)$ and  $\sum_i |E^i| =m$. 
\end{enumerate} 
\end{lemma}
 
We are now ready to show existence of a minimal partition of  $\R^n$. 
 
 \begin{theorem}\label{thex} 
 For all $m>0$ there exists a lattice $G\in \G_m$ and a fundamental domain $D\in\D_G$ such that
 \[\Per_\phi(D)=\min_{G\in\G_m,\,E\in\D_G} \Per_\phi(E). \]
 \end{theorem} 
 \begin{proof} Let $C=\inf_{G\in\G_m} \inf_{E\in\D_G} \Per_\phi(E)$. If $C=+\infty$ there is nothing to prove; if $C>0$, we consider a minimizing sequence $G_k$ of lattices in $\G_m$, and of fundamental domains $D_k$ associated to $G_k$, with $\Per_\phi(D_k)\leq 2C$. Then by Lemma \ref{Fubini}, up to a subsequence $G_k\to G$ in Kuratowski sense, where $G$ is a lattice in $\G_m$.  We consider a reduced set of generators, which exists due to  Lemma \ref{limi}, $v_k^i$ of $G_k$:  we get that $|v_k^1|\dots |v_k^n|\leq C_n m$. So, 
 since $|v_k^i|\geq \delta$ for all $i$ and $k$, there holds that $|v_k^i|\leq C_n m\, \delta^{1-n}:=R(m,\delta)$ for all $i$ and $k$. 
 
 By Lemma \ref{lemmacc}, there exist $z_k^i\in \R^n$  and $E^i$ measurable sets in $\R^n$, such that $D_k-z_k^i\to E^i$, $|z_k^i-z_k^j|\to +\infty$ as $k\to +\infty$, $\sum_i |E^i| =m$ and  
 $$\sum_i\Per_\phi(E^i)\leq  \liminf_k\Per_\phi(D_k).$$ 
 We observe that for all $i$, we may choose $g_k^i\in G_k$, with $|g_k^i-z_k^i|\leq \frac{\sqrt n}{2}R(m, \delta)$, so that $|g_k^i-g_k^j|\to +\infty$ as $k\to +\infty$, when $i\neq j$.

 Possibly passing to a subsequence,
 we get that there exist  measurable sets in $D^i$ such that $D_k-g_k^i\to D^i$,  $\sum_i |D^i| =m$ and 
 $$\sum_i\Per_\phi(D^i)\leq  \liminf_k\Per_\phi(D_k)=\inf_{G\in\G_m} \inf_{E\in\D_G} \Per_\phi(E).$$ 
 
In order to conclude, we have to prove that $D:=\cup_i D^i$ is a fundamental domain for the limit group $G$. First of all we prove that $|D^i\cap( D^j+g)|=0$ for all $i\neq j$, and for all $g\in G$, and that $|(D^i+g)\cap D^i|=0$ for all $g\in G$, $g\neq 0$.  This would imply, together with the fact that $\sum_i |D^i|=m$, that $\cup_i D^i=D$ is a fundamental domain associated to $G$ (see \cite[Lemma 2.2]{cnparti}), and then that $D$ is a solution to the isoperimetric problem \eqref{iso}. 

Notice that, since $D_k$ are fundamental domains, we have 
$$|D_k \cap (D_k-g_k^j+g_k^i)|=0 \qquad \text{for all $i\neq j$.}$$
Passing to the limit as $k\to +\infty$, we then get $|D^i\cap D^j|=0$. 

Let now $g\in G$, with $g\neq 0$. So $g=\lim h_k$, for some $h_k\in G_k$, with $h_k\neq 0$. Therefore, $D_k-g_k^i+h_k\to D^i+g$ in $L^1_{loc}$.  Moreover, since $D_k$ are fundamental domains associated to $G_k$, there holds $|(D_k-g_k^i+h_k)\cap (D_k-g_k^j)|=0$ for all $k, i,j$. 
Passing to the limit as $k\to +\infty$, we then get $|(D^i+g)\cap D^j|=0$ for all $g\neq 0$, $i,j$.   
 \end{proof} 
 
We now state  the regularity of minimizers of \eqref{iso}, which is a consequence of the fact that the $G$-periodic partition generated by a minimal fundamental domain is a local minimizer of the perimeter functional. For the proof we refer to \cite[Theorem 4.8, Theorem 4.9]{cnparti}.  
 
 \begin{corollary} 
 Let $D\in \D_G$ be a solution to problem \eqref{iso} given by Theorem \ref{thex}, and assume that $\phi^2$ is uniformly convex and of class $C^2$.
 Then $\partial D$ is a $C^{1,\alpha}$ hypersurface, for some $\alpha\in (0,1)$, up to a closed singular set $\Sigma\neq \emptyset$ 
with $\mathcal{H}^{n-1}(\Sigma)=0$. 
 Moreover, in the isotropic case $\phi(x)=|x|$, the $G$-periodic partition generated by a minimal fundamental domain $D$ is locally finite, $D$ is bounded, 
 and $\partial D\setminus \Sigma$ is of class $C^\infty$. Finally, if $n=2$, then $D$ is a centrally symmetric hexagon.
 \end{corollary} 
 
 We recall the notions of decomposable and indecomposable set (see \cite{ac}).
 
 \begin{definition}
 Let $E$ be a set of finite perimeter. $E$ is decomposable if there exist $A, B\subset E$ subsets such that $|A|, |B|>0$, $|A|+|B|=|E|$
 and $\Per(E)=\Per(A)+\Per(B)$. We say that $E$ is indecomposable if it is not decomposable. \end{definition}   
 
 \begin{corollary} 
Let $\phi(x)=|x|$ and let $D\in \D_G$ be a solution to problem \eqref{iso}. Then $\R^n\setminus D$ is indecomposable.
If $n\in\{2,3\}$, then $D$ is also  indecomposable.
\end{corollary} 
 
 \begin{proof}
 Let us assume by contradiction that $\R^n\setminus D$ is decomposable.
 By regularity, both $D$ and $\R^n\setminus D$ are union of a finite number of indecomposable components. Moreover, 
since $D$ is bounded,  there exists only one unbounded indecomposable component of $\R^n\setminus D$ and all the others are bounded. 
Let $A$ be a bounded component of $\R^n\setminus D$. Since $D$ is a fundamental domain,
the component $A$ is covered by translations of indecomposable components of  $D$, that is,
$A=\cup_{i=1}^N D_i+g_i$, with $g_i\in G$. 

Let now $n\in\{2,3\}$, and assume by contradiction that we can write $D=D_1\cup D_2$, with $|D_1|, |D_2|>0$, $|D_1|+|D_2|=|D|$
 and $\Per(D_1)+\Per(D_2)=\Per(D)$. Observe $D_g=D_1\cup (D_2+g)$ is a fundamental domain 
 for all  $g\in G$, and there exists $g$ such that $\partial D_1\cap \partial(D_2+g)\ni x\in\R^n$.
 In particular $D_g$ induces a minimal partition which has a disconnected phase in $B_r(x)$ (given by $D_g\cap B_r(x)$) for all $r>0$.
 This is in contradiction with the complete classification of singular point which is available in dimensions $2, 3$ (see \cite{taylor}).
\end{proof} 
  
An interesting question, which remains unsolved, is if a minimal domain $D$ is necessarily indecomposable or if its interior is homeomorphic to a ball.

\section{Some remarks in the isotropic case} \label{sec:rem}
We briefly discuss some candidate solutions to the isoperimetric problem \eqref{iso}, in the isotropic case $\phi(x)=|x|$.

First of all, we recall that the problem admits a unique solution in $\R^2$ due to Hales \cite{hales} (see also \cites{morgansolo,morgan}), given by a regular hexagon. We mention that the result by Hales is much stronger, since it gives uniqueness of the minimizer in a large class of competitors.

\begin{theorem}[Hales honeycomb theorem] 
Any partition of $\R^2$ into regions of equal area has average perimeter at least that of the regular hexagonal honeycomb tiling.
\end{theorem} 

In three dimensions, the variational problem corresponding to the surface minimizing partition
into cells of equal volume is known as the  Kelvin problem \cite{kelvin}. 
Lord Kelvin formulated a conjecture about the explicit shape of the minimizer, based on the bitruncated cubic honeycomb:
a small deformation of the faces produces a minimal partition, which is Kelvin's proposed solution, see \cite{we} for an more explicit description of this cell.
Weaire and Phelan \cite{w} provided  a three-dimensional foam of equal-volume cells, containing two different shapes, 
and  strictly less surface area than the Kelvin's partition.  However, it remains unknown if Kelvin's foam is a minimizer among foams with equal cells, 
or if it is a solution to problem \eqref{iso}.

\smallskip

In $\R^4$ there is a well known regular tessellation, which is  the $24$-cell honeycomb.
The $24$-cell  is the convex hull of its vertices which can be described as the $24$ coordinate permutations of $(\pm 1, \pm1, 0,0)$.
 In this frame of reference the $24$-cell has edges of length $\sqrt{2}$ and is inscribed in a $3$-sphere of radius $\sqrt{2}$. 
 Up to a rescaling, it is also the Voronoi cell of the $D_4$ lattice.
 
If a sphere is inscribed in each $24$ cell of the $24$-cell honeycomb, the resulting arrangement is the densest known regular sphere packing in four dimensions, with kissing number $24$, even if the sphere packing problem  is still unsolved in dimension $4$.  
This suggests that the $24$-cell could be a solution to problem \eqref{iso}. 

In this direction, we observe that 
 the $24$-cell honeycomb satisfies the following local minimality property, which is a necessary condition for being a solution to problem \eqref{iso}.

\begin{proposition} 
The  $24$-cell honeycomb  is a minimizer for the perimeter in every ball $B_r(x)$, with $x\in \R^4$ and $0<r<\sqrt{2}$.  
\end{proposition} 

\begin{proof} 
The thesis follows from the fact that inside a ball $B_r(x)$ with $r<\sqrt{2}$, the $24$-cell honeycomb is either empty or coincides, up to a translation, with 
one of the following cones in $\R^4$: a hyperplane, three half-hyperplanes meeting at 120 degrees, 
the cone over the $1$-skeleton of a tetrahedron times $\R$, the cone over the $2$-skeleton of a hypercube.
Thanks to the results by Taylor \cite{taylor} and Brakke  \cite{brakke} it is known that such cones define partitions of $\R^4$ which are minimal for the perimeter under compact perturbations.
\end{proof}  
In analogy with the case of the $24$-cell in $\R^4$, one may wonder if the Voronoi cells associated to the $E_8$ lattice in $\R^8$ and the Leech lattice in $\R^{24}$, which are solutions to the  sphere packing problem and the kissing number problem (see \cites{c,v}), are also solutions to problem \eqref{iso}. 

\smallskip

A reduced problem which should be easier to tackle is given by 
\begin{equation}\label{isobis}
\inf_{G\in\G_m} \Per_\phi(V_G),
\end{equation} 
 where $V_G$ is the Voronoi cell of the lattice $G$.
 
The existence of a minimal lattice for \eqref{isobis} can be shown reasoning as in Section \ref{sec:ex} above.
In this case, it should be simpler to explore the optimality of the $D_4$, $E_8$ and Leech lattices, respectively.
In three dimensions one might expect that the BCC lattice is the optimal one, since its Voronoi cell is given by the truncated octahedron
(see \cite[Conjecture 2.1]{Gal14}).

\medskip
Finally we discuss the asymptotic behavior as the dimension $n$ goes to $+\infty$  of the isoperimetric function,  that is 
   the minimal value in \eqref{iso}, with $m=1$. 
We have the following result.
\begin{proposition} \label{asi} The isoperimetric function $c(n)$ satisfies:
\begin{equation}\label{cn}
\sqrt{2\pi e n}\sim n\omega_n^{\frac{1}{N}}\leq c(n)\leq   \frac{ 2 }{(2\zeta(n))^{\frac{1}{n}}}  n \omega_n^{\frac{1}{n}} \sim 2 \sqrt{2\pi e n}
\end{equation}
where $\zeta(n)$ is the Riemann zeta-function. 
\end{proposition} 

\begin{proof}
We consider a generic lattice $G\in \G_1$ and its associated Voronoi cell $V_G$  (see Definition \ref{voronoi}). 
Let  $F_i$ be the facets of $V_G$ and $h_i$ the distance between the center of the cell and the hypersurface cointaing   $F_i$: by its definition we get that $h_i\geq \rho_G$, where $\rho_G$ is the packing radius of $G$. Then we obtain 
\[
1=|V_G|=\frac{\sum_i  \Per(F_i) h_i}{n}\ge \frac{\Per(V_G)\rho_G}{n}\qquad\text{ and so } \Per(V_G)\leq \frac{n}{\rho_G}. 
\]
 Minkowski-Hlawka theorem \cite[chapter 1]{CS}  provides a nonconstructive proof of the existence of lattices $G$ in $\R^N$ such that 
 \[
 \rho_G\geq  \frac{\zeta(n)^{\frac{1}{n}}}{\omega_n^{\frac{1}{n}} 2^{1-\frac{1}{n}}}, \] where $\zeta(n)$ is the Riemann zeta-function. 
As a consequence, for these lattices $G$ we get 
\[ \Per(V_G)\leq \frac{n}{\rho_G}\leq \frac{n \omega_n^{\frac{1}{n}} 2^{1-\frac{1}{n}}}{\zeta(n)^{\frac{1}{n}}}=\frac{2}{(2\zeta(n))^{\frac{1}{n}}}n\omega_n^{\frac{1}{n}}\sim 2 \sqrt{2\pi e n}.\]
On the other side by direct comparison with the perimeter of the ball of volume $1$, we get  $\Per (V_G)\geq n\omega_n^{\frac{1}{N}}\sim \sqrt{2\pi e n}$. This permits to conclude  \eqref{cn}. 
  \end{proof}



\begin{bibdiv}
\begin{biblist}

\bib{ac}{article}{
    AUTHOR = {Ambrosio, L.},
    author={Caselles, V.},
    author={Masnou, S.},
    author={Morel, J.M.}, 
    title={Connected Components of Sets of Finite Perimeter and Applications to Image Processing}, 
journal={J. Eur. Math. Soc.(JEMS)},
volume={23},
year={2001}, 
number={1},
pages={39–92},
}

\bib{brakke}{article}{
AUTHOR = {Brakke, K. A.},
     TITLE = {Minimal cones on hypercubes},
   JOURNAL = {J. Geom. Anal.},
    VOLUME = {1},
      YEAR = {1991},
    NUMBER = {4},
     PAGES = {329--338},
}
		
\bib{cassels}{book}{ 
    AUTHOR = {Cassels, J. W. S.},
     TITLE = {An introduction to the geometry of numbers},
    SERIES = {Classics in Mathematics},
      NOTE = {Corrected reprint of the 1971 edition},
 PUBLISHER = {Springer-Verlag, Berlin},
      YEAR = {1997},
     PAGES = {viii+344},
      ISBN = {3-540-61788-4},}
	
\bib{cnparti}{article}{
    AUTHOR = {Cesaroni, A.},
    author={Novaga, M.},
     TITLE = {Periodic partitions with minimal perimeter},
   JOURNAL = {ArXiv Preprint 2212.11545},
      YEAR = {2022},
   }
     
     
\bib{choe}{article}{
   AUTHOR = {Choe, J.},
     TITLE = {On the existence and regularity of fundamental domains with
              least boundary area},
   JOURNAL = {J. Differential Geom.},
    VOLUME = {29},
      YEAR = {1989},
    NUMBER = {3},
     PAGES = {623--663},
}

\bib{c}{article}{
 AUTHOR = {Cohn, Henry},
 author={Kumar, Abhinav},  
 author={Miller, Stephen D.},
 author={Radchenko, Danylo}, 
 author={Viazovska, Maryna},
     TITLE = {Universal optimality of the {$E_8$} and {L}eech lattices and
              interpolation formulas},
   JOURNAL = {Ann. of Math. (2)},
     VOLUME = {196},
      YEAR = {2022},
    NUMBER = {3},
     PAGES = {983--1082},
}
\bib{CS}{book}{
author={J.H. Conway},
author={N.J.A. Sloane}, 
 TITLE = {Sphere packings, lattices and groups. Third edition},
    SERIES = {Grundlehren der mathematischen Wissenschaften},
    VOLUME = {290},
 PUBLISHER = {Springer-Verlag, New York},
      YEAR = {1999},
}
		
\bib{Gal14}{article}{ 
AUTHOR = {Gallagher, Paul},
 author={Ghang, Whan},  
 author={Hu, David},
 author={Martin, Zane}, 
 author={Miller, Maggie},
 author={Perpetua, Byron}, 
 author={Waruhiu, Steven},
 title={Surface-area-minimizing n-hedral Tiles},
journal={Rose-Hulman Undergraduate Mathematics Journal}, 
volume={15},
year={2014},
NUMBER = {1},
PAGES = {Article 13},
}

\bib{hales}{article}{
author={ Hales, T.C.},
title={The honeycomb conjecture},
journal={Discrete Comput. Geom.},
volume={ 25},
 YEAR = {2001},
    NUMBER = {1},
     PAGES = {1--22},
}

\bib{maggi}{book}{
    AUTHOR = {Maggi, F.},
     TITLE = {Sets of finite perimeter and geometric variational problems},
    SERIES = {Cambridge Studies in Advanced Mathematics},
    VOLUME = {135},
 PUBLISHER = {Cambridge University Press, Cambridge},
      YEAR = {2012},
}

\bib{ma}{article}{
AUTHOR = {Mahler, K.},
     TITLE = {On lattice points in {$n$}-dimensional star bodies. {I}.
              {E}xistence theorems},
   JOURNAL = {Proc. Roy. Soc. London Ser. A},
     VOLUME = {187},
      YEAR = {1946},
     PAGES = {151--187},
      }		
      
\bib{mnpr}{article}{
    AUTHOR = {Martelli, B.}, 
    AUTHOR = {Novaga, M.}, 
    AUTHOR = {Pluda, A.},
    AUTHOR = {Riolo, S.},
     TITLE = {Spines of minimal length},
   JOURNAL = {Ann. Sc. Norm. Super. Pisa Cl. Sci. (5)},
    VOLUME = {17},
      YEAR = {2017},
    NUMBER = {3},
     PAGES = {1067--1090},
     }
      
\bib{morgan}{article}{
    AUTHOR = {Morgan, F.},
    AUTHOR = {French, C.},
    AUTHOR = {Greenleaf, S.},
     TITLE = {Wulff clusters in {$\mathbb R^2$}},
   JOURNAL = {J. Geom. Anal.},
    VOLUME = {8},
      YEAR = {1998},
    NUMBER = {1},
     PAGES = {97--115},
}

\bib{morgansolo}{article}{
    AUTHOR = {Morgan, F.},
     TITLE = {The hexagonal honeycomb conjecture},
   JOURNAL = {Trans. Amer. Math. Soc. },
    VOLUME = {351},
      YEAR = {1999},
    NUMBER = {5},
     PAGES = {1753--1763},
}

		
\bib{npst}{article}{
    AUTHOR = {Novaga, M.}, 
    author={Paolini, E.}, 
    author={Stepanov, E.}, 
    author={ Tortorelli, V.M.},
     TITLE = {Isoperimetric clusters in homogeneous spaces via concentration compactness},
   JOURNAL = { J. Geom. Anal. },
    VOLUME = {32},
      YEAR = {2022},
      NUMBER = {11},
      PAGES = {Paper No. 263}, 
}


\bib{taylor}{article}{
    AUTHOR = {Taylor, J. E.},
     TITLE = {The structure of singularities in solutions to ellipsoidal
              variational problems with constraints in {${\rm R}^{3}$}},
   JOURNAL = {Ann. of Math. (2)},
    VOLUME = {103},
      YEAR = {1976},
    NUMBER = {3},
     PAGES = {541--546},
}

\bib{thompson}{article}{ 
AUTHOR = {Thomson (Lord Kelvin), W.},
     TITLE = {On the division of space with minimum partitional area},
   JOURNAL = {Acta Math.},
    VOLUME = {11},
      YEAR = {1887},
    NUMBER = {1-4},
     PAGES = {121--134},
}
\bib{v}{article}{ 
    AUTHOR = {Viazovska, Maryna S.},
     TITLE = {The sphere packing problem in dimension $8$},
   JOURNAL = {Ann. of Math. (2)},
    VOLUME = {185},
      YEAR = {2017},
    NUMBER = {3},
     PAGES = {991--1015},
}
	
\bib{kelvin}{book}{
EDITOR = {Weaire, D.},
     TITLE = {The {K}elvin problem},
         NOTE = {Foam structures of minimal surface area},
 PUBLISHER = {Taylor \& Francis, London},
      YEAR = {1996},
      }
 
 \bib{we}{article}{
author={ D. Weaire}, 
title={Kelvin's foam structure: a commentary},
journal={ Philosophical
Magazine Letters}, 
volume={88},
number={2},
pages={ 91--102},
year={2008},}

\bib{w}{article}{
author = {Weaire, D.}, 
author={Phelan, R.},
title = {A counter-example to Kelvin's conjecture on minimal surfaces},
journal = {Philosophical Magazine Letters},
volume = {69},
number = {2},
pages = {107-110},
year  = {1994},
} 
\end{biblist}\end{bibdiv}
 \end{document}